\documentclass[a4paper,12pt]{amsart}

\usepackage{amssymb}
\usepackage{amsmath}
\usepackage{amscd}
\usepackage{graphicx}
\usepackage{amsthm}
\usepackage{color}
\usepackage[all]{xy}

\theoremstyle{definition}
\newtheorem{defn}{Definition}[section]
\newtheorem{thm}[defn]{Theorem}
\newtheorem{cor}[defn]{Corollary}

\newtheorem{lem}[defn]{Lemma}
\newtheorem{ex}[defn]{Example}
\newtheorem{fact}[defn]{Fact}

\newtheorem{rem}[defn]{Remark}
\newtheorem{cla}[defn]{Claim}

\newtheorem{introthm}{Theorem}
\newtheorem{introcor}{Corollary}

\newcommand{\mf}[1]{{\mathfrak{#1}}}
\newcommand{\mb}[1]{{\mathbf{#1}}}
\newcommand{\bb}[1]{{\mathbb{#1}}}

\newcommand{\ceil}[1]{\lceil #1 \rceil}
\newcommand{\floor}[1]{\lfloor #1 \rfloor}

\DeclareMathOperator{\fpt}{fpt}

\DeclareMathOperator{\len}{length}

\DeclareMathOperator{\rank}{rank}

\begin{document}

\title{Diagonal F-thresholds and F-pure thresholds of Hibi rings}
\author{Takahiro Chiba}
\address[Takahiro Chiba]{Graduate School of Mathematics, Nagoya University,
         Nagoya 464--8602, Japan}
\email{m08023c@math.nagoya-u.ac.jp}
\author{Kazunori Matsuda}
\address[Kazunori Matsuda]{Graduate School of Mathematics, Nagoya University,
         Nagoya 464--8602, Japan}
\email{d09003p@math.nagoya-u.ac.jp}
\date{\today}


\begin{abstract}
Hibi rings are a kind of graded toric ring on a finite distributive lattice $D = J(P)$, where $P$ is a partially ordered set. 
In this paper, we compute diagonal $F$-thresholds and $F$-pure thresholds of Hibi rings. 
\end{abstract}

\maketitle


\section*{Introduction}


\par
In this paper, we study two invariants of commutative Noetherian rings of positive characteristic, that is, the $F$-threshold and the $F$-pure threshold. 

In \cite{MTW}, Musta\c{t}\u{a}, Takagi and Watanabe defined the notion of $F$-thresholds for $F$-finite regular local rings. 
And in \cite{HMTW}, Huneke, Musta\c{t}\u{a}, Takagi and Watanabe generalized it in more general situation. 

In higher dimensional algebraic geometry over a field $k$ of characteristic $0$, the log canonical thresholds are important objects.
In \cite{TW}, Takagi and Watanabe introduced the notion of the $F$-pure threshold, which is an analogue of the log canonical threshold.

Firstly, we recall the definition of $F$-threshold. 
Let $(R, \mf{m})$ be an $F$-finite $F$-pure local domain or a standard graded $k$-algebra with the unique graded maximal ideal $\mf{m}$, of characteristic $p > 0$. 
Then the following limit value exists(see \cite{HMTW}):
\[
c^{\mf{m}}(\mf{m}) = \lim_{e \to \infty} \frac{\max\{r \in \bb{N} \mid \mf{m}^{r} \not\subset \mf{m}^{[p^{e}]}\}}{p^{e}}, 
\]
where $\mf{m}^{[p^{e}]} = (x^{p^{e}} \mid x \in \mf{m})$. 
We call it the {\em diagonal F-threshold} of $R$. 

Secondly, we recall the definition of $F$-pure threshold. 
Let $t \ge 0$ be a real number and $\mf{a}$ a nonzero ideal of $R$. 
The pair $(R, \mf{a}^{t})$ is said to be {\em F-pure} if for all large $q = p^{e} \gg 0$, there exists an element $d \in \mf{a}^{\floor{t(q - 1)}}$ such that $R \to R^{1/q} (1 \mapsto d^{1/q})$ splits as an $R$-linear map. 
Then the {\em F-pure threshold}, denoted by $\fpt(\mf{a})$, is defined by

\begin{center}
$\fpt(\mf{a}) = \sup\{t \in \bb{R}_{\ge 0} \mid$ the pair $(R, \mf{a}^{t})$ is $F$-pure$\}$. 
\end{center}

There are a few examples of these invariants.
Hence it seems to be important to compute $F$-thresholds and $F$-pure thresholds concretely for several rings. 
In \cite{MOY}, the second author, Ohtani and Yoshida computed diagonal $F$-thresholds and $F$-pure thresholds for binomial hypersurfaces.

In this paper, we pick up Hibi rings. 
We compute diagonal $F$-thresholds $c^{\mf{m}}(\mf{m})$ and $F$-pure thresholds $\fpt(\mf{m})$ of such rings and describe these invariants in terms of poset. 

Let $P$ be a finite partially ordered set(poset for short), and $\mathcal{R}_{k}[D]$ the Hibi ring over a field $k$ of characteristic $p > 0$ on a distributive lattice $D = J(P)$, where $J(P)$ is the set of all poset ideals of $P$. 

The main theorem in this paper is the following:

\begin{introthm}(see Theorem 2.4, Theorem 3.9 and Corollary 4.2)
Let $R = \mathcal{R}_{k}[D]$ be a Hibi ring, and $\mf{m} = R_{+}$ the unique graded maximal ideal of $R$. Then
	\begin{align*}
		c^{\mf{m}}(\mf{m}) &= \rank^{*} P + 2,\\
		-a(R) &= \rank P + 2,\\
		\fpt(\mf{m}) &= \rank_{*} P + 2.
	\end{align*}
In particular, 
\begin{enumerate}
	\item $c^{\mf{m}}(\mf{m})$, $\fpt(\mf{m}) \in \bb{N}$, 
	\item $\fpt(\mf{m}) \le \min\{\len C \mid C \in \mathcal{C}\} + 2 \le -a(R) = \max\{\len C \mid C \in \mathcal{C}\} + 2$
	$\le c^{\mf{m}}(\mf{m})$, 
	where $a(R)$ denotes the $a$-invariant of $R$ (see \cite{GW}) and $\mathcal{C}$ denotes the set of all maximal chains of $P$.
\end{enumerate}
\end{introthm}

Recently, this inequality $\fpt(\mf{m}) \le -a(R) \le c^{\mf{m}}(\mf{m})$ was proved by Hirose, Watanabe and Yoshida for any homogeneous toric ring $R$ (see \cite{HWY}). 

In \cite{Hir}, Hirose gave formulae of $F$-thresholds and $F$-pure thresholds for any homogeneous toric ring $R$. 
However, it seems to be difficult for us to construct many examples by his formula. 
For Hibi rings, we give formulae of $c^{\mf{m}}(\mf{m})$ and $\fpt(\mf{m})$ in terms of poset, that is, the upper rank (denoted by $\rank^{*}$) and the lower rank (denoted by $\rank_{*}$). 
Thanks to these formulae, we can find enough examples. 
More precisely, for given integers $a \ge b\ge c \ge 1$, we can find a connected poset $P$ such that $\rank^{*} P = a, \rank P = b$ and $\rank_{*} P = c$ (see Example 4.4). 

As a corollary, we give formulae of $c^{\mf{m}}(\mf{m})$ and $\fpt(\mf{m})$ of Segre products of two polynomial rings. 
Segre products are important objects in commutative ring theory and combinatorics. 

\begin{introcor}
Let $k$ be a perfect field of positive characteristic, and let
$m,n \ge 2$ be integers. 
Let $R=k[X_1, \ldots, X_m], S=k[Y_1, \ldots, Y_n]$ be polynomial rings, and let $R \# S$ be the Segre product of $R$ and $S$. 
Let $\mf{m}$ be the unique graded maximal ideal of $R \# S$. 
Then
	\begin{align*}
		c^{\mf{m}}(\mf{m}) = -a(R \# S) &= \max\{m, n\},\\
		\fpt(\mf{m}) &= \min\{m, n\}. 
	\end{align*}
In particular, $c^{\mf{m}}(\mf{m}) = \fpt(\mf{m})$ if and only if $m = n$.
\end{introcor}


Let us explain the organization of this paper.
In Section 1, we set up the notions of posets, and define the Hibi ring and $\rank^{*} P$ and $\rank_{*} P$ in order to state our main theorem.

In Section 2, we recall the definition and several basic results of $F$-threshold and give a formula of diagonal $F$-thresholds $c^{\mf{m}}(\mf{m})$ for Hibi rings.

In Section 3, we recall the definition of $F$-pure threshold and give a formula of $F$-pure thresholds $\fpt(\mf{m})$ for Hibi rings.  

In Section 4, we compute a-invariants $a(R)$ for Hibi rings and compare $c^{\mf{m}}(\mf{m})$ and $\fpt(\mf{m})$ with $-a(R)$. 
Moreover, for given integers $a \ge b\ge c \ge 1$, we find a connected poset $P$ such that $\rank^{*} P = a, \rank P = b$ and $\rank_{*} P = c$. 


\section{Preliminaries}
\par
First, we set up the notions of posets and define the Hibi ring.


Let $P = \{p_{1}, p_{2}, \ldots, p_{N}\}$ be a finite partially ordered set(poset for short). 
Let $J(P)$ be the set of all poset ideals of $P$, where a poset ideal of $P$ is a subset $I$ of $P$ such that if $x \in I$, $y \in P$ and $y \le x$ then $y \in I$.
By structure theorem of Birkhoff(see \cite{Bir}), for a distributive lattice $D$, there exists a poset $P$ such that $D \cong J(P)$ ordered by inclusion. 

A {\em chain} $X$ of $P$ is a totally ordered subset of $P$.
The {\em length} of a chain $X$ of $P$ is $\#X - 1$, where $\#X$ is the cardinality of $X$.
The {\em rank} of $P$, denoted by $\rank P$, is the maximum of the lengths of chains in $P$.
A poset is called {\em pure} if its all maximal chains have the same length.
For $x, y \in P$, we say that $y$ {\em covers} $x$, denoted by $x \lessdot y$, if $x < y$ and there is no $z \in P$ such that $x < z < y$.


\begin{defn}[\cite{Hib}]
Let the notation be as above. 
We consider the following map:
\[
\varphi : D( = J(P)) \hspace{5mm} \longrightarrow \hspace{5mm} K[T, X_{1}, \ldots, X_{N}]
\]

\hspace{45mm} \rotatebox{90}{$\in$} \hspace{42mm} \rotatebox{90}{$\in$}

\[
\hspace{8mm} I \hspace{16mm} \longmapsto \hspace{15mm} T \prod_{p_{i} \in I} X_{i} 
\]
Then we define the {\em Hibi ring} $\mathcal{R}_{k}[D]$ as follows:
\[
\mathcal{R}_{k}[D] = k[\varphi(I) \mid I \in J(P)].
\]
\end{defn}


\begin{ex}
Consider the following poset $P(1 \le 3, 2 \le 3 \ {\rm and}\ 2 \le 4)$:

$\ $

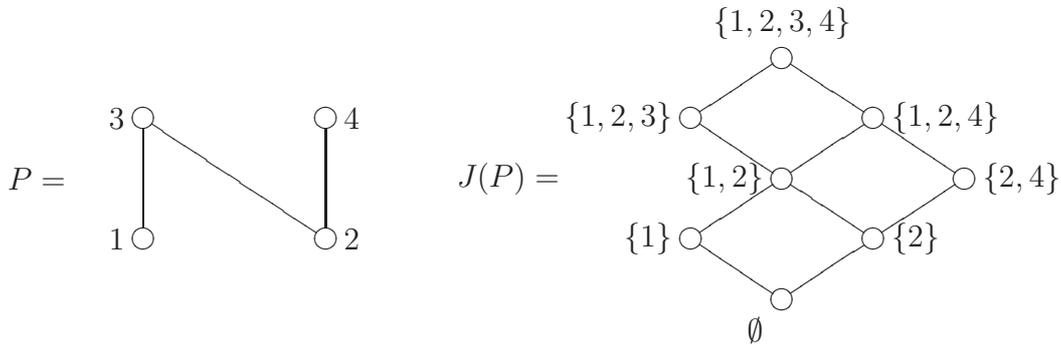
\begin{xy}
	\ar@{} (0, 0); (10, -16) = "A",
	\ar@{} "A" *{P = }; (24, -24) *++!R{1} *\cir<4pt>{} = "B",
	\ar@{-} "B"; (24, -8) *++!R{3} *\cir<4pt>{} = "C",
	\ar@{-} "C"; (48, -24) *++!L{2} *\cir<4pt>{} = "D",
	\ar@{-} "D"; (48, -8) *++!L{4} *\cir<4pt>{} = "E",
	\ar@{} "E"; (72, -16) *{J(P) = } = "F",
	\ar@{} "F"; (108, -16) *++!R{\{1, 2\}} *\cir<4pt>{} = "G",
	\ar@{} "G"; (132, -16) *++!L{\{2, 4\}} *\cir<4pt>{} = "H",
	\ar@{-} "G"; (120, -8) *++!L{\{1, 2, 4\}} *\cir<4pt>{} = "I",
	\ar@{-} "H"; "I",
	\ar@{} "I"; (120, -24) *++!L{\{2\}} *\cir<4pt>{} = "J",
	\ar@{-} "G"; "J",
	\ar@{-} "H"; "J",
	\ar@{} "I", (96, -8) *++!R{\{1, 2, 3\}} *\cir<4pt>{} = "K",
	\ar@{-} "G"; "K",
	\ar@{} "K", (108, 0) *++!D{\{1, 2, 3, 4\}} *\cir<4pt>{} = "L",
	\ar@{-} "I"; "L",
	\ar@{-} "K"; "L",
	\ar@{-} "G"; (96, -24) *++!R{\{1\}} *\cir<4pt>{} = "M",
	\ar@{-} "M"; (108, -32) *++!RU{\emptyset} *\cir<4pt>{} = "N",
	\ar@{-} "J"; "N",
	\end{xy}
Then we have
\[
\mathcal{R}_{k}[D] = k[T, TX_{1}, TX_{2}, TX_{1}X_{2}, TX_{2}X_{4}, TX_{1}X_{2}X_{3}, TX_{1}X_{2}X_{4}, TX_{1}X_{2}X_{3}X_{4}].
\]	
\end{ex}

\begin{ex}
Consider the following poset $P$:

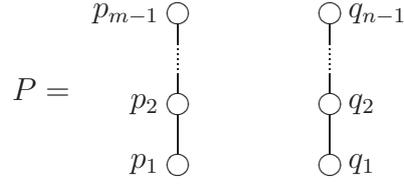
\begin{xy}
	\ar@{} (0, 0); (42, -8) = "A",
	\ar@{} "A" *{P = }; (60, -18) *++!R{p_{1}} *\cir<4pt>{} = "B",
	\ar@{-} "B"; (60, -10) *++!R{p_{2}} *\cir<4pt>{} = "C",
	\ar@{-} "C"; (60, -6) = "D",
	\ar@{.} "D"; (60, -2) = "E",
	\ar@{-} "E"; (60, 2) *++!R{p_{m - 1}} *\cir<4pt>{}= "F",
	\ar@{} "F"; (80, -18) *++!L{q_{1}} *\cir<4pt>{} = "G",
	\ar@{-} "G"; (80, -10) *++!L{q_{2}} *\cir<4pt>{} = "H",
	\ar@{-} "H"; (80, -6) = "I",
	\ar@{.} "I"; (80, -2) = "J",
	\ar@{-} "J"; (80, 2) *++!L{q_{n - 1}} *\cir<4pt>{}	
\end{xy}

then $\mathcal{R}_{k}[D] \cong k[X]/I_{2}(X)$, where $X$ is an $m \times n$-matrix whose all entries are indeterminates. 

\end{ex}


\begin{rem}
	\begin{enumerate}
		\item (\cite{Hib}) Hibi rings are toric ASL, thus normal Cohen-Macaulay domains. 
		\item $\dim \mathcal{R}_{k}[D] = \# P + 1$. 
		\item (\cite{Hib}) $\mathcal{R}_{k}[D]$ is Gorenstein if and only if $P$ is pure. 
	\end{enumerate} 
\end{rem}


Finally, we define $\rank^{*} P$ and $\rank_{*} P$ for a poset $P$ in order to state our main theorem.

A sequence $C = (q_{1}, \ldots, q_{t})$ is called a {\em path} of $P$ if $C$ satisfies the following conditions: 

\begin{enumerate}
	\item $q_{1}, \ldots, q_{t}$ are distinct elements of $P$,
	\item $q_{1}$ is a minimal element of $P$ and $q_{t - 1} \lessdot q_{t}$,	
	\item $q_{i} \lessdot q_{i + 1}$ or $q_{i + 1} \lessdot q_{i}$.
	\end{enumerate}
In short, we regard the Hasse diagram of $P$ as a graph, and consider paths on it.
In particular, if $q_{t}$ is a maximal element of $P$, then we call $C$ {\em maximal} path.
For a path $C = (q_{1}, \ldots, q_{t})$, we denote $C = q_{1} \to q_{t}$. 

For a path $C = (q_{1}, \ldots, q_{t})$, $q_{i}$ is said to be a {\em locally maximal element} of $C$ if $q_{i - 1} \lessdot q_{i}$ and $q_{i + 1} \lessdot q_{i}$, and a {\em locally minimal element} of $C$ if $q_{i} \lessdot q_{i - 1}$ and $q_{i} \lessdot q_{i + 1}$. 
For convenience, we consider that $q_{1}$ is a locally minimal element and $q_{t}$ is a locally maximal element of $C$. 

For a path $C = (q_{1}, \ldots, q_{t})$, if $q_{1} \le \cdots \le q_{t}$ then we call $C$ an {\em ascending chain} and if $q_{1} \ge \cdots \ge q_{t}$ then we call $C$ a {\em descending chain}. 
We denote a ascending chain by a symbol $A$ and a descending chain by a symbol $D$.
For a ascending chain $A = (q_{1}, \ldots, q_{t})$, we put $t(A) = q_{t}$ and $<A> = \{q \in P \mid q \le t(A)\}$. 
Since $<A>$ is a poset ideal of $P$ generated by $A$, we note that $<A> \in J(P)$.

Let $C = (q_{1}, \ldots, q_{t})$ be a path. 
We now introduce the notion of the {\em decomposition} of $C$. 
We decompose $V(C)$ as follows:
\[
V(C) = V(A_{1}) \cup V(D_{1}) \cup V(A_{2}) \cup \cdots \cup V(D_{n - 1}) \cup V(A_{n})
\]
such that
\begin{align*}
	V(A_{1}) &= \{q_{1}, \ldots, q_{a(1)}\}, \\	
	V(D_{1}) &= \{q^{\prime}_{1}, \ldots, q^{\prime}_{d(1)}\}, \\
	V(A_{2}) &= \{q_{a(1) + 1}, \ldots, q_{a(2)}\}, 
\end{align*}

\vspace{2mm}

\begin{center}	
$\rotatebox{90}{$\cdots$}$
\end{center}
\begin{align*}
	\hspace{8mm} V(D_{n - 1}) &= \{q^{\prime}_{d(n - 2) + 1}, \ldots, q^{\prime}_{d(n - 1)}\}, \\
	V(A_{n}) &= \{q_{a(n - 1) + 1}, \ldots, q_{a(n)} = q_{t}\},
\end{align*}
where $\{q_{a(1)}, \ldots, q_{a(n)}\}$ is the set of locally maximal elements and $\{q_{1}, q^{\prime}_{d(1)}, \ldots, q^{\prime}_{d(n - 1)}\}$ is the set of locally minimal elements of $C$. 
Then $A_{i}$ are ascending chains and $D_{j}$ are descending chains.
This decomposition is denoted by $C = A_{1} + D_{1} + A_{2} + \cdots + D_{n - 1} + A_{n}$.

For a path $C = (q_{1}, \ldots, q_{t})$, we define {\em the upper length} by 
\[
\len^{*} C = \#\{(q_{i}, q_{i + 1}) \in E(C) \mid q_{i} \lessdot q_{i + 1}\}, 
\]
where $E(C)$ is the set of edges of $C$. 

\begin{ex}
	\begin{enumerate}
		\item If $C$ is a chain, then $\len^{*} C = \len C$.
		\item Consider the following path $C$:
		
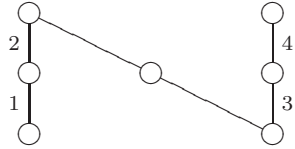
			\begin{xy}
				\ar@{} (64, -8) *++!D{} *\cir<4pt>{} = "A";	
				\ar@{-}_{1} "A";(64, 0) *++!D{} *\cir<4pt>{} = "B";
				\ar@{-}_{2} "B";(64, 8) *++!D{} *\cir<4pt>{} = "C";
				\ar@{-} "C";(80, 0)	*++!D{} *\cir<4pt>{} = "D";
				\ar@{-} "D";(96, -8) *++!D{} *\cir<4pt>{} = "E";
				\ar@{-}^{3} "E";(96, 0) *++!D{} *\cir<4pt>{} = "F";
				\ar@{-}_{4} "F";(96, 8) *++!D{} *\cir<4pt>{}		
			\end{xy}
		
		\vspace{5mm}
		
		\noindent Then $\len^{*} C = 4$. 
	\end{enumerate}

\end{ex}


Next, we introduce the condition (*). 

\begin{defn}
For a path $C = (q_{1}, \ldots, q_{t})$, we say that $C$ {\em satisfies a condition} (*) if $C$ satisfies the following conditions: 

For the decomposition $C = A_{1} + D_{1} + \ldots + D_{n - 1} + A_{n}$, 

(1) $V(D_{i}) \cap \left(\bigcup^{i - 1}_{m = 1} <A_{m}> \cup <A_{i} \setminus t(A_{i})> \cup \{t(A_{i})\}\right) = \emptyset$,

(2) $V(A_{i + 1}) \cap \left(\bigcup^{i}_{m = 1} <A_{m}> \right) = \emptyset$

\noindent for all $1 \le i \le n - 1$. 
\end{defn}

\begin{rem}
At the above definition, condition (*) means as follows: 
assume that $C = (q_{1}, \ldots, q_{r - 1}, q_{r}, q_{r + 1}, \ldots)$ and $q_{r}$ is a locally maximal element or a locally minimal element of $C$. 
Then for all $s > r$ and $r > t$, $q_{s} \not\le q_{t}$.  
\end{rem}

\begin{rem}
For a path $C = (q_{1}, \ldots, q_{t})$ such that $C$ satisfies a condition (*) and $q_{t}$ is a locally maximal element, we can extend $C$ to a path $\tilde{C} = (q_{1}, \ldots, q_{t}, \ldots, q_{t^{\prime}})$ such that $\tilde{C}$ is a maximal path which satisfies a condition (*). 

Indeed, if $q_{t}$ is not a maximal element of $P$, then there exists $q_{t + 1}$ such that $q_{t} \lessdot q_{t + 1}$. 
We decompose $C = A_{1} + D_{1} + \ldots + D_{n - 1} + A_{n}$. 
If $q_{t + 1} \in \ <A_{i}>$ for some $i$, then so is $q_{t}$. 
This means that $C$ does not satisfy a condition (*), a contradiction. 
Hence a path $C^{\prime} = (q_{1}, \ldots, q_{t}, q_{t + 1})$ also satisfies a condition (*). 
Therefore, by repeating this operation, we can extend $C$ to a path $\tilde{C} = (q_{1}, \ldots, q_{t}, \ldots, q_{t^{\prime}})$ such that $\tilde{C}$ is a maximal path which satisfies a condition (*). 
\end{rem}

\begin{ex}
Consider the following poset $P$: 

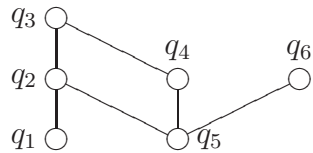
	\begin{xy}
		\ar@{} (64, -8) *++!R{q_{1}} *\cir<4pt>{} = "A";
		\ar@{-} "A";(64, 0) *++!R{q_{2}} *\cir<4pt>{} = "B";
		\ar@{-} "B";(64, 8) *++!R{q_{3}} *\cir<4pt>{} = "C"
		\ar@{-} "C";(80, 0)	*++!D{q_{4}} *\cir<4pt>{} = "D";
		\ar@{-} "D";(80, -8) *++!L{q_{5}} *\cir<4pt>{} = "E";
		\ar@{-} "B";"E";
		\ar@{-} "E";(96, 0) *++!D{q_{6}} *\cir<4pt>{}		
	\end{xy}
	
	\vspace{5mm}
	
\noindent Then, $C_{1} = (q_{1}, q_{2}, q_{5}, q_{6})$ satisfies the condition (*), but $C_{2} = (q_{1}, q_{2}, q_{3}, q_{4}, q_{5}, q_{6})$ does not satisfy the condition (*) because $q_{2} \ge q_{5}$.
\end{ex}

Now, we define the upper rank $\rank^{*} P$ and the lower rank $\rank_{*} P$ for a poset $P$. 

\begin{defn}
For a poset $P$, we define
	\begin{align*}
		\rank^{*} P &= \max\{\len^{*} C \mid C {\rm\ is\ a\ maximal\ path\ which\ satisfies\ a\ condition} (*)\},\\
		\rank_{*} P &= \min\{\len^{*} C \mid C {\rm\ is\ a\ maximal\ path\ which\ satisfies\ a\ condition} (*)\}.
	\end{align*}
We call $\rank^{*} P$ {\em upper rank} and $\rank_{*} P$ {\em lower rank} of $P$.
We note that $\rank^{*} P \ge \rank P \ge \rank_{*} P$.
\end{defn}

\begin{ex}
Consider the following poset $P$:

	\begin{xy}
		\ar@{} (72, -8) *++!R{q_{1}} *\cir<4pt>{} = "A";
		\ar@{-} "A";(72, 0) *++!R{q_{2}} *\cir<4pt>{} = "B";
		\ar@{-} "B";(72, 8) *++!R{q_{3}} *\cir<4pt>{} = "C"
		\ar@{} "C";(88, 0)	*++!L{q_{5}} *\cir<4pt>{} = "D";
		\ar@{-} "D";(88, -8) *++!L{q_{4}} *\cir<4pt>{} = "E";
		\ar@{-} "B";"E";
		\ar@{-} "D";(88, 8) *++!L{q_{6}} *\cir<4pt>{}			
	\end{xy}

\vspace{3mm}

\noindent Then, the following paths satisfy the condition (*):

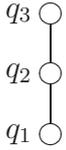
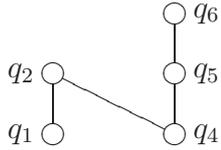
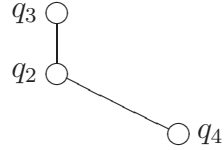
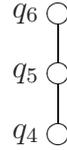
	\begin{xy}
		\ar@{} (0, -8) *++!R{q_{1}} *\cir<4pt>{} = "A";
		\ar@{-} "A";(0, 0) *++!R{q_{2}} *\cir<4pt>{} = "B";
		\ar@{-} "B";(0, 8) *++!R{q_{3}} *\cir<4pt>{} = "C";
		\ar@{} "A";(16, -8)
		
		\ar@{} (16, -8);(32, -8) *++!R{q_{1}} *\cir<4pt>{} = "D";
		\ar@{-} "D";(32, 0) *++!R{q_{2}} *\cir<4pt>{} = "E";
		\ar@{-} "E";(48, -8) *++!L{q_{4}} *\cir<4pt>{} = "F";
		\ar@{-} "F";(48, 0) *++!L{q_{5}} *\cir<4pt>{} = "G";
		\ar@{-} "G";(48, 8) *++!L{q_{6}} *\cir<4pt>{} = "H";
		\ar@{} "F";(64, -8)
		
		\ar@{} (64, -8);(80, 0) *++!R{q_{2}} *\cir<4pt>{} = "I";
		\ar@{-} "I";(80, 8) *++!R{q_{3}} *\cir<4pt>{} = "J";
		\ar@{-} "I";(96, -8) *++!L{q_{4}} *\cir<4pt>{} = "K";
		\ar@{} "K";(112, -8)
		
		\ar@{} (112, -8);(128, -8) *++!R{q_{4}} *\cir<4pt>{} = "L";
		\ar@{-} "L";(128, 0) *++!R{q_{5}} *\cir<4pt>{} = "M";
		\ar@{-} "M";(128, 8) *++!R{q_{6}} *\cir<4pt>{} = "N";
		\ar@{} "N";(136, -8)
				
		\end{xy}
		
\vspace{3mm}

\noindent Hence we have $\rank^{*} P = 3$ and $\rank_{*} P = \rank P = 2$. 
			
\end{ex}


\section{Diagonal F-thresholds of Hibi rings}
\par
In this section, we recall the definition and several basic results 
of $F$-threshold and give a proof of the main theorem.

\subsection{Definition and basic properties}

Let $R$ be a Noetherian ring of characteristic $p > 0$ with $\dim R = d \ge 1$.
Let $\mf{m}$ be a maximal ideal of $R$. 
Suppose that $\mf{a}$ and $J$ are $\mf{m}$-primary ideals of $R$ such that $\mf{a} \subseteq \sqrt{J}$ and $\mf{a} \cap R^{\circ} \neq \emptyset$, where $R^{\circ}$ is the set of elements of $R$ that are not contained in any minimal prime ideal of $R$. 


\begin{defn}
Let $R, \mf{a}, J$ be as above. For each nonnegative integer $e$, put $\nu_{\mf{a}}^{J}(p^{e}) = \max \{ r \in \bb{N} \mid \mf{a}^{r} \not\subseteq J^{[p^{e}]}\}$, where $J^{[p^{e}]} = (a^{p^{e}} \mid a \in J)$.
Then we define 
\[
c^{J}(\mf{a}) = \lim_{e \to \infty} \frac{\nu_{\mf{a}}^{J}(p^{e})}{p^{e}}
\]
if it exists, and call it the $F$-{\em threshold} of the pair $(R, \mf{a})$ with respect to $J$. 
Moreover, we call $c^{\mf{a}}(\mf{a})$ the {\em diagonal} $F$-{\em threshold} of $R$ with respect to $\mf{a}$. 

For convenience, we put
\[
c_{+}^{J}(\mf{a)} = \limsup_{e \to \infty} \frac{\nu_{\mf{a}}^{J}(p^{e})}{p^{e}}, \hspace{6mm} c_{-}^{J}(\mf{a)} = \liminf_{e \to \infty} \frac{\nu_{\mf{a}}^{J}(p^{e})}{p^{e}}.
\]
\end{defn}

About basic properties and examples of $F$-thresholds, see \cite{HMTW}.
In this section, we summarize basic properties of the diagonal $F$-thresholds $c^{\mf{m}}(\mf{m})$. 


\begin{ex}
	\begin{enumerate}
		\item Let $(R, \mf{m})$ be a regular local ring of positive characteristic. Then $c^{\mf{m}}(\mf{m}) = \dim R.$
		\item Let $k[X_{1}, \ldots, X_{d}]^{(r)}$ be the $r$-th Veronese subring of a polynomial ring $S = k[X_{1}, \ldots, X_{d}]$. Put $\mf{m} = (X_{1}, \ldots, X_{d})^{r}R$. Then $c^{\mf{m}}(\mf{m}) = \frac{r + d - 1}{r}$.
		 
		\item ([MOY, Corollary 2.4]) If $(R, \mf{m})$ is a local ring with $\dim R = 1$, then $c^{\mf{m}}(\mf{m}) = 1$.
		
	\end{enumerate}
\end{ex}

\begin{ex}([MOY, Theorem 2])
Let $S = k[X_{1}, \ldots, X_{m}, Y_{1}, \ldots, Y_{n}]$ be a polynomial ring  over $k$ in $m + n$ variables, and put $\mf{n} = (X_{1}, \ldots, X_{m}, Y_{1}, \ldots, Y_{n})S$. 
Take a binomial $f = X_{1}^{a_{1}} \cdots X_{m}^{a_{m}} - Y_{1}^{b_{1}} \cdots Y_{n}^{b_{n}} \in S$, where $a_{1} \ge \cdots \ge a_{m}, b_{1} \ge \cdots \ge b_{n}$. 
Let $R = S_{\mf{n}}/(f)$ be a binomial hypersurface local ring with the unique maximal ideal $\mf{m}$. Then
\[
c^{\mf{m}}(\mf{m}) = m + n - 2 + \frac{\max\{a_{1} + b_{1} - \min\{\sum_{i = 1}^{m} a_{i}, \sum_{j = 1}^{n} b_{j}\}, 0\}}{\max\{a_{1}, b_{1}\}}.
\]		
\end{ex}


\subsection{Proof of the main theorem}
\par
In this subsection, we give a proof of the main theorem. 
Recall Theorem 1:

\begin{thm}
Let $P$ be a finite poset, and $D = J(P)$ the distributive lattice. Put $R = \mathcal{R}_{k}[D]$. Let $\mf{m} = R_{+}$ be the graded maximal ideal of $R$. Then
\[
c^{\mf{m}}(\mf{m}) = \rank^{*} P + 2.
\]  
\end{thm}


\begin{lem}
$c_{-}^{\mf{m}}(\mf{m}) \ge \rank^{*} P + 2$.
\end{lem}

\begin{proof}
First of all, we note that for all $Q = p^{e}$, 
\[
\mf{m}^{[Q]} = (\varphi(I)^{Q} \mid I \in J(P))R
\]
and
\[
R = \bigoplus^{+\infty}_{r = 0}\left( T^{r}X_{1}^{s(p_{1})} \cdots X_{N}^{s(p_{N})} \mid 0 \le s(p_{i}) \le r, p_{i} \le p_{j} \Rightarrow s(p_{i}) \ge s(p_{j}) \right). 
\]

Take a path $C$ such that $\len^{*} C = \rank^{*} P$ and decompose $C = A_{1} + D_{1} + \cdots + D_{n - 1} + A_{n}$, where
\begin{align*}
	V(A_{1}) &= \{q_{1}, \ldots, q_{a(1)}\}, \\	
	V(D_{1}) &= \{q^{\prime}_{1}, \ldots, q^{\prime}_{d(1)}\}, \\
	V(A_{2}) &= \{q_{a(1) + 1}, \ldots, q_{a(2)}\}, 
\end{align*}

\vspace{2mm}

\begin{center}	
$\rotatebox{90}{$\cdots$}$
\end{center}
\begin{align*}
	\hspace{8mm} V(D_{n - 1}) &= \{q^{\prime}_{d(n - 2) + 1}, \ldots, q^{\prime}_{d(n - 1)}\}, \\
	V(A_{n}) &= \{q_{a(n - 1) + 1}, \ldots, q_{a(n)} = q_{m}\}.
\end{align*}
Then we note that $m = \rank^{*} P + 1$. 

Next, we define an increasing sequence of poset ideals as follows:
\[
I_{1} = \{q_{k(1)}\},
\]
\[
I_{i} = <\{q_{k(i)}\}> \cup I_{i - 1} \ \ (2 \le i \le m). 
\]

To prove Lemma 2.5, it is enough to show that

\begin{cla}
\[
M = \prod_{I = \emptyset, I_{1}, \ldots, I_{m}} \varphi(I)^{Q - 1} \in \mf{m}^{(m + 1)(q - 1)} \setminus \mf{m}^{[Q]}. 
\]
\end{cla}

\noindent{\em Proof of Claim 2.6.}
Put $M = T^{r} X^{s(p_{1})}_{1} \cdots X^{s(p_{n})}_{n}$. 
Then, by the construction of $M$, we have $r = (m + 1)(Q - 1)$. 
Hence $M \in \mf{m}^{(m + 1)(Q - 1)}$. 
Moreover, since $C$ satisfies a condition (*), $s(q_{k(1)}) = m(Q - 1)$, $s(q_{k(2)}) = (m - 1)(Q - 1)$, \ldots, $s(q_{k(m)}) = Q - 1$. 

We assume that $M \in \mf{m}^{[Q]}$. 
Then there exists $I \in J(P)$ such that $M/\varphi(I)^{Q} \in R$. 
Put $M/\varphi(I)^{Q} = T^{r^{\prime}} X^{s^{\prime}(p_{1})}_{1} \cdots X^{s^{\prime}(p_{n})}_{n}$. 
Then $0 \le s^{\prime}(p_{i}) \le r^{\prime}$ and $s^{\prime}(p_{i}) \le s^{\prime}(p_{j})$ if $p_{i} \ge p_{j}$. 

By the construction of $M$, we have $r^{\prime} = m(Q - 1) - 1$. 
Moreover, $s^{\prime}(p_{i}) = s(p_{i}) - Q$ if $p_{i} \in I$ and $s^{\prime}(p_{i}) = s(p_{i})$ if $p_{i} \not\in I$. 
Hence, if $q_{k(1)} \not\in I$, then $s^{\prime}(q_{k(1)}) = s(q_{k(1)}) = m(Q - 1) < r$, a contradiction. 
Therefore $q_{k(1)} \in I$. 
Also if $q_{k(2)} \not\in I$, then $s^{\prime}(q_{k(2)}) = s(q_{k(2)}) = (m - 1)(Q - 1) > (m - 1)(Q - 1) - 1 = s^{\prime}(q_{k(1)})$. 
This contradicts $q_{k(2)} \ge q_{k(1)}$. 
Hence $q_{k(2)} \in I$. 
In the same way, $q_{k(3)}, \ldots, q_{k(m)} \in I$. 
This contradicts $s(q_{k(m)}) = Q - 1 < Q$. 
Therefore $M \not\in \mf{m}^{[Q]}$. 
\end{proof}

Hence we have $c^{\mf{m}}(\mf{m}) \ge \rank^{*} P + 2$. 
Next, we prove the opposite inequality. 

\begin{lem}
For all large $Q = p^{e} \gg 0$, if $r \ge (\rank^{*} P + 2)(Q - 1) + 1$ then $\mf{m}^{r} \subseteq \mf{m}^{[Q]}$.
\end{lem}

\begin{proof}
We note that
\[
\mf{m}^{r} = (T^{r}X_{1}^{s(p_{1})} \cdots X_{N}^{s(p_{N})} \mid 0 \le s(p_{i}) \le r, p_{i} \ge p_{j} \Rightarrow s(p_{i}) \le s(p_{j}))R. 
\]
We will show that for each $M = T^{r}X_{1}^{s(p_{1})} \cdots X_{N}^{s(p_{N})} \in \mf{m}^{r}$, there exists $I \in J(P)$ such that $\frac{M}{\varphi(I)^{Q}} \in \mathcal{R}_{k}[D]$.

\noindent Case 1: For all minimal elements $p \in P$, $r - s(p) \ge Q$. 

Put $I = \emptyset$. 
Then $\frac{M}{\varphi(I)^{Q}} \in \mathcal{R}_{k}[D]$. 
Indeed, put $\frac{M}{\varphi(I)^{Q}} = T^{r^{\prime}} X^{s^{\prime}(p_{1})}_{1} \cdots X^{s^{\prime}(p_{n})}_{n}$, then $r^{\prime} = r - Q$ and $s^{\prime}(p_{i}) = s(p_{i})$. 
Hence $0 \le s^{\prime}(p_{i}) \le r^{\prime}$ and $s^{\prime}(p_{i}) \le s^{\prime}(p_{j})$ if $p_{i} \ge p_{j}$.  
Therefore $\frac{M}{\varphi(I)^{Q}} \in \mathcal{R}_{k}[D]$.

\noindent Case 2: There exists a minimal element $p \in P$ such that $r - s(p) \le Q - 1$. 

For each $p \in P$, we define a function $d_{M}:P \to \{0, 1\}$ as follows: 
We define $d_{M}(p) = 1$ if there exists a path $C = p_{\min} \to p$ such that $C$ satisfies the following conditions, and $d_{M}(p) = 0$ otherwise: 

\begin{enumerate}
	\item $r - s(p_{\min}) \le Q - 1$.  
	\item $C$ satisfies a condition (*). 
	\item We decompose $C = A_{1} + D_{1} + A_{2} + \cdots + D_{n^{\prime} - 1} + A_{n^{\prime}}$. 
	Let $q_{1}, \ldots, q_{m^{\prime}}$ be the elements of $V(A_{1}). \ldots, V(A_{n^{\prime}})$ as in Lemma 2.5. 
	Then for all $i = 1, \ldots, m^{\prime}$, $s(q_{k(i)}) - s(q_{k(i + 1)}) \le Q - 1$. 
\end{enumerate}

\begin{fact}
\begin{enumerate}
	\item If $d_{M}(p) = 1$ then $s(p) \ge Q$. 
	\item If $p^{\prime} \gtrdot p$, $d_{M}(p) = 1$ and $d_{M}(p^{\prime}) = 0$, then $s(p) - s(p^{\prime}) \ge Q$. 
\end{enumerate}
\end{fact}

\begin{proof}[Proof of Fact 2.8]
(1) For all $p \in P$ such that $d_{M}(P) = 1$, by definition of $d_{M}$, there exists a path $C = p_{\min} \to p$ such that $C$ satisfies the following conditions:

\begin{enumerate}
	\item $r - s(p_{\min}) \le Q - 1$.  
	\item $C$ satisfies a condition (*). 
	\item We decompose $C = A_{1} + D_{1} + A_{2} + \cdots + D_{n^{\prime} - 1} + A_{n^{\prime}}$. 
	Let $q_{1}, \ldots, q_{m^{\prime}}$ be the elements of $V(A_{1}). \ldots, V(A_{n^{\prime}})$ as in Lemma 2.5. 
	Then for all $i = 1, \ldots, m^{\prime}$, $s(q_{k(i)}) - s(q_{k(i + 1)}) \le Q - 1$. 
\end{enumerate}
Then we note that $\len^{*} C = m^{\prime} + 1 \le m + 1$. 
We put 

\begin{align*}
	V(A_{1}) &= \{q_{1}, \ldots, q_{a(1)}\}, \\	
	V(D_{1}) &= \{q^{\prime}_{1}, \ldots, q^{\prime}_{d(1)}\}, \\
	V(A_{2}) &= \{q_{a(1) + 1}, \ldots, q_{a(2)}\}, 
\end{align*}

\vspace{2mm}

\begin{center}	
$\rotatebox{90}{$\cdots$}$
\end{center}
\begin{align*}
	\hspace{8mm} V(D_{n - 1}) &= \{q^{\prime}_{d(n - 2) + 1}, \ldots, q^{\prime}_{d(n - 1)}\}, \\
	V(A_{n}) &= \{q_{a(n - 1) + 1}, \ldots, q_{a(n)} = q_{m}\}.
\end{align*}

Since $r \ge (m + 1)(Q - 1) + 1$ by assumption, we get
\begin{align*}
	s(q_{m}) &\ge r - m^{\prime}(Q - 1) \\
		     &\ge (m + 1)(Q - 1) + 1 - m^{\prime}(Q - 1) \\
		     &\ge Q.
\end{align*}

(2) Assume that $p^{\prime} \gtrdot p$, $d_{M}(p) = 1$ and $d_{M}(p^{\prime}) = 0$. 
If $s(p) - s(p^{\prime}) \le Q - 1$, then there exists a path $C = p_{\min} \to p$ since $d_{M}(p) = 1$. 
By Remark 1.8, we can extend $C$ to a path $\tilde{C} = p_{\min} \to p^{\prime}$ satisfying a condition (*). 
Hence $d_{M}(p^{\prime}) = 1$, a contradiction. 
Therefore $s(p) - s(p^{\prime}) \ge Q$. 
\end{proof}

We return to the proof of Theorem 2.4. 
Put $I = \{p \in P \mid$ there exists $p^{\prime} \ge p$ such that $d_{M}(p^{\prime}) = 1\} \in J(P)$.  
We prove that $M/\varphi(I)^{Q} \in \mathcal{R}_{k}[D]$. 

Put $M/\varphi(I)^{Q} = T^{r^{\prime}} X^{s^{\prime}(p_{1})}_{1} \cdots X^{s^{\prime}(p_{N})}_{N}$. 
Then $r^{\prime} = r - Q$. 
Moreover, $s^{\prime}(p_{i}) = s(p_{i}) - Q$ if $p_{i} \in I$ and $s^{\prime}(p_{i}) = s(p_{i})$ if $p_{i} \not\in I$. 

Firstly, we prove that $s^{\prime}(p_{i}) \le s^{\prime}(p_{j})$ if $p_{i} \ge p_{j}$. 
We may assume that $p_{i} \gtrdot p_{j}$. 
We note that $s^{\prime}(p_{i}) = s(p_{i}) - Q$ if $p_{i} \in I$ and $s^{\prime}(p_{i}) = s(p_{i})$ if $p_{i} \not\in I$. 
Hence, if $p_{i}, p_{j} \in I$, or $p_{i}, p_{j} \not\in I$, then $M/\varphi(I)^{Q} \in \mathcal{R}_{k}[D]$. 
Therefore, we may assume that $p_{i} \not\in I$ and $p_{j} \in I$. 
Then we have $d_{M}(p_{i}) = 0$. 
If $d_{M}(p_{j}) = 1$, then $s^{\prime}(p_{i}) =s(p_{i}) \le s(p_{j}) - Q = s^{\prime}(p_{j})$ by Fact 2.8(2). 
If $d_{M}(p_{j}) = 0$, then there exists $p_{k} \in P$ such that $d_{M}(p_{k}) = 1$ and $p_{k} \ge p_{j}$. 
If $p_{k} \ge p_{i}$, then $p_{i} \in I$, a contradiction. 
Hence $p_{k} \not\ge p_{i}$. 
Since $d_{M}(p_{k}) = 1$, there exists a path $C = p_{\min} \to p_{k}$. 

\noindent Case 2-1: We can extend $C$ to a path $\tilde{C} = p_{\min} \to p_{i}$ satisfying a condition (*). 

If $s(p_{k}) - s(p_{i}) \le Q - 1$, then $d_{M}(p_{i}) = 1$, a contradiction. 
Hence $s(p_{k}) - s(p_{i}) \ge Q$. 
Therefore, we have $s^{\prime}(p_{j}) - s^{\prime}(p_{i}) = s(p_{j}) - Q - s(p_{i}) \ge s(p_{k}) - s(p_{i}) - Q \ge 0$. 

\noindent Case 2-2: We cannot extend $C$ as Case 2-1. 

In this case, a path $\tilde{C} = p_{\min} \to p_{i}$ does not satisfy a condition (*). 
Hence there exists $p_{\ell} \in V(C)$ such that $p_{\ell} \ge p_{k}, p_{i}$. 
This contradicts $d_{M}(p_{j}) = 0$. 
Therefore, we have that $s^{\prime}(p_{i}) \le s^{\prime}(p_{j})$ if $p_{i} \ge p_{j}$. 

Secondly, we prove that $0 \le s^{\prime}(p_{i}) \le r^{\prime}$. 
By Fact 2.8(1), $0 \le s^{\prime}(p_{i})$.  
To prove $s^{\prime}(p_{i}) \le r^{\prime}$, it is enough to show that $s^{\prime}(p_{\min}) \le r^{\prime}$ for all minimal element $p_{\min}$. 
If $p_{\min} \in I$, $s^{\prime}(p_{\min}) = s(p_{\min}) - Q \le r - Q = r^{\prime}$. 
Assume that $p_{\min} \not\in I$. 
If $r - s(p_{\min}) \le Q - 1$, then $d_{M}(p) = 1$, a contradiction. 
Hence $r - s(p_{\min}) \ge Q$, and thus $r^{\prime} - s^{\prime}(p_{\min}) = r - Q - s(p_{\min}) \ge 0$.   
\end{proof}

\begin{cor}
$c_{+}^{\mf{m}}(\mf{m}) \le \rank^{*} P + 2$.
\end{cor}

As a result, we have $c^{\mf{m}}(\mf{m}) = \rank^{*} P + 2$.


\section{$F$-pure~thresholds of Hibi rings}

The $F$-pure~threshold, 
which was introduced by \cite{TW}, 
is an invariant of an ideal of an $F$-finite $F$-pure ring. 
$F$-pure~threshold can be calculated by computing generalized test ideals (see \cite{HY}),
and \cite{Bl} showed how to compute generalized test ideals in the case of toric rings and its monomial ideals.
Since Hibi rings are toric rings, 
we can compute $F$-pure~thresholds of the homogeneous maximal ideal of arbitrary Hibi rings,
and will be described in terms of poset.

\begin{defn}[$F$-pure~threshold\cite{TW}]
	Let $R$ be an $F$-finite $F$-pure ring of characteristic $p > 0$, 
	$\mf{a}$ a nonzero ideal of $R$, 
	and $t$ a non-negative real number.
	The pair $(R, \mf{a}^t)$ is said to be $F$-pure if for all large $q = p^e$, 
	there exists an element $d \in \mf{a}^{\ceil{t(q-1)}}$ such that 
	the map $R \longrightarrow R^{1/q} \ (1 \mapsto d^{1/q})$ splits as an $R$-linear map. 
	Then the $F$-pure~threshold $\fpt(\mf{a})$ is defined as follows: 
	\begin{align*}
		\fpt(\mf{a}) = \sup\{t \in \mathbb{R}_{\geq 0} \mid (R, \mf{a}^t)  \text{ is } F \text{-pure}\}. 
	\end{align*}
\end{defn}

Hara and Yoshida \cite{HY} introduced the generalized test ideal $\tau(\mf{a}^t)$ ($t$ is a non negative real number).
Then $\fpt(\mf{a})$ can be calculated as the minimum jumping number of $\tau(\mf{a}^c)$, that is, 

\begin{align*}
	\fpt(\mf{a}) = \sup \{ t \in R_{\geq 0} \mid \tau(\mf{a}^t) = R \}.
\end{align*}

Especially, \cite{Bl} showed how to calculate $\tau(\mf{a}^c)$ in the case of monomial ideals $\mathfrak{a}$ in a toric ring $R$.  
Now, we recall the following theorem of \cite{Bl}. 

\subsection{setting for toric rings}
Let $k$ be a perfect field, 
$N = M^{\vee} \cong \mathbb{Z}^n$ a dual pair of lattices. 
Let $\sigma \subset N_{\mathbb{R}} = N \otimes_{\mathbb{Z}} \mathbb{R}$ be a strongly convex rational polyhedral cone 
given by $\sigma = \{r_1u_1 + \cdots + r_su_s \mid r_i \in \mathbb{R}_{\geq 0} \}$ for some $u_1, \dots, u_s$ in $N$. 
The dual cone $\sigma^{\vee}$ is a (rational convex polyhedral) cone in $M_{\mathbb{R}}$ defined by
$\sigma^{\vee} = \{m \in M_{\mathbb{R}} \mid (m,v) \geq 0, \forall v \in \sigma\}$. 
The lattice points in $\sigma^{\vee}$ give a sub-semigroup ring of Laurent polynomial ring $k[X_1^{\pm 1}, \dots, X_n^{\pm 1}]$,
generated by $X^m = X_1^{m_1} \cdot \cdots \cdot X_n^{m_n} \ (m \in \sigma^{\vee} \cap M)$.
This affine semigroup ring is denoted by

\begin{align*}
	R_{\sigma} = k[\sigma^{\vee} \cap M].
\end{align*}

$R_{\sigma}$ is said to be the toric ring defined by ${\sigma}$.
For a monomial ideal $\mf{a} = (X^{\alpha_1}, \dots, X^{\alpha_s}) \subset R_{\sigma}$, 
$P(\mf{a})$ denotes the convex hull of $\alpha_1, \dots, \alpha_s$ in $M_{\mathbb{R}}$.

\begin{thm}[\cite{Bl}]
	Let $R_{\sigma}$ be a toric ring defined by ${\sigma}$ over a field of positive characteristic and $\mf{a}$ a monomial ideal of $R_{\sigma}$. 
	Let $v_1, \dots, v_s \in \mathbb{Z}^n$ be the primitive generator of $\sigma$.
	Then a monomial $X^m \in R_{\sigma}$ is in $\tau(\mf{a}^c)$ if and only if 
	there exists $w \in M_{\mathbb{R}}$ with $(w,v_i) \leq 1$ for all $i$, such that
	\begin{align*}
		m + w \in \text{relint } cP(\mf{a}),
	\end{align*}
	where $cP(\mf{a}) = \{cm \in M \mid m \in P(\mf{a})\}$.
\end{thm}

Especially, $\tau(\mf{a}^c) = R$ if and only if $X^0 \in \tau(\mf{a}^c)$.
Let $$\mathcal{O} = \{ w \in M_{\mathbb{R}} \mid (w,v_i) \geq 1 (\exists i)\}.$$
Then we get following corollary.

\begin{cor}[\cite{Hir}]
\label{hi}
	$\fpt(\mf{a}) = \sup \{ c \in \mathbb{R}_{\geq 0} \mid 
	(\check{\sigma} \setminus \mathcal{O}) \cap cP(\mf{a}) \neq \emptyset \}$. 
\end{cor}

\subsection{$F$-pure threshold of Hibi rings}
Since Hibi rings are toric rings, 
we can compute $F$-pure~threshold of the homogeneous maximal ideal of any Hibi ring using corollary \ref{hi}. 

Recall that Hibi rings have the structure as toric rings. 
Let $P$ be a finite poset, $R = R_k(D), 
\mf{m}$ the unique homogeneous maximal ideal of $R$, where $D=J(P)$. 
$\mathbb{R}^P$ denotes $\#P$-dimensional $\mathbb{R}$-vector space, 
which entries are indexed by $P$. 
$\mathbb{Z}^P$ denotes the lattice points in $\mathbb{R}^P$. 
For a monomial $T^{u_T}\prod_{p \in P}X_p^{u_p} \in R$, 
$u = (u_T, u_p)_{p \in P}$ is a corresponding vector in $\mathbb{Z} \oplus \mathbb{Z}^P$. 
It is known that $R$ is a toric ring defined by a strongly convex rational polyhedral cone generated from "the order polytope of $P$".

\begin{defn}[cf. \cite{St}]
	$P, \mathbb{R}^P$ are as above. 
	An element of P descrived as $(u_p)_{p \in P}$. 
	The order polytope of $P$ is a subset of $\mathbb{R}^P$ satisfying following conditions. 
	\begin{enumerate}
		\item[1)] $0 \leq u_p \leq 1$ for all $p \in P$. 
		\item[2)] $u_p \leq u_{p'}$ if $p \geq p'$. 
	\end{enumerate}
\end{defn}

\begin{rem}
Note that the condition 2) is slightly different from the original. 
It is arranged for construction of Hibi rings in this paper. 
\end{rem}

Let $\mathfrak{m}$ be the maximal homogeneous ideal of $R$. 
>From a constraction of $R$, $P(\mathfrak{m}) - (1, \overrightarrow{0}) \subset (0) \oplus \mathbb{R}^P$ is the order polytope of $P$.

\begin{lem}
	$$ R = k[\mathbb{R}_{\geq 0} P(\mathfrak{m}) \cap (\mathbb{Z} \oplus \mathbb{Z}^P)].$$ 
	Hence, if we put $\sigma^{\vee} = \mathbb{R}_{\geq 0} P(\mathfrak{m})$, 
	then $R$ is the toric ring defined by $\sigma$. 
\end{lem}


Now, the primitive generators of $\check{\sigma}$ is the following.

\begin{align*}
	\left\{ u = (u_i, u_T)_{i \in P} \left| 
	\begin{array}{cc}
		u_T = 1 \\
		u_i = 1 & (i \in I)\\
		u_i = 0 & (i \not \in I)
	\end{array}
	,I \in J(P)
	\right. \right\}. 
\end{align*}

And $R$ is represented as $k[X^u \mid u \in \check{\sigma} \cap \mathbb{Z}^{\#P +1}]$.
Since $P(\mf{m}) = \check{\sigma} \cap (u_T = 1)$, we can obtain the following lemma from Corollary \ref{hi}.

\begin{align}
	\label{Hfpt}
	\fpt(\mf{m}) = \sup \{ \deg_T u \mid u \in \check{\sigma} \setminus \mathcal{O} \}.
\end{align}

Set $\overline{P}$ be $P \cup \{-\infty, \infty\}$, 
and let $\Sigma$ be the set of real functions $\psi$
satisfying following properties.

\begin{description}
	\item[1)] $\psi(\infty)=0$.
	\item[2)] $x \lessdot y \Longrightarrow \psi(y) - \psi(x) \leq 1$
\end{description}

\begin{thm}
	\label{hhh}
	Let $R = \mathcal{R}_k(D)$ be the Hibi ring corresponding to a finite poset $P$, 
	and $\mf{m}$ its homogeneous maximal ideal. 
	Then 
	\begin{align*}
		\fpt(\mf{m}) = \max \{\psi(-\infty) \mid \psi \in \Sigma \}. 
	\end{align*}
\end{thm}

\begin{proof}
	Note that $\check{\sigma}$ is a rational polyhedral cone given by the following conditions: 
	
	\begin{align*}
		\begin{array}{ccc}
			0 \leq & u_i & (i : \text{maximal}), \\
			0 \leq & u_i - u_j & (j \lessdot i), \\
			0 \leq & u_i - u_T & (i : \text{minimal}). 
		\end{array}
	\end{align*}
	
	Then, $\check{\sigma} \setminus \mathcal{O}$ is a domain satisfying following conditions. 
	
	\begin{align*}
		\begin{array}{ccccc}
			0 \leq & u_i & \leq & 1 & (i ; \text{maximal}), \\
			0 \leq & u_i - u_j & \leq & 1 & (j \lessdot i), \\
			0 \leq & u_i - u_T & \leq & 1 & (i ; \text{minimal}). 
		\end{array}
	\end{align*}
	
	Thus we obtain the required assertion. 
 
\end{proof}

 From the theorem~\ref{hhh}, we get the following inequality. 

\begin{cor}
	\begin{align*}
		\fpt(\mf{m}) 
		&\leq \min \{\len C \mid C : \text{maximal chain in} \ \overline{P} \} \\
		&= \min \{\len C \mid C : \text{maximal chain in} \ P \}+2. \\
	\end{align*}
\end{cor}

We have another assertion of $\fpt(\mathfrak{m})$ in terms of $\rank_*P$. 

\begin{thm}
	\label{fptrank}
	Under the same notation of theorem~\ref{hhh}, 
	
	$$\fpt(\mathfrak{m})=\rank_* \overline{P} = \rank_*P + 2. $$ 
	
	In particular, $$\fpt(\mf{m}) \in \mathbb{N}. $$ 
\end{thm}

\begin{proof}
	The second equality is clear. 
	First, we prove that $\fpt(\mathfrak{m}) \leq \rank_*P +2$. 
	If $A = (p_0, \dots, p_r)$ is a chain in $\overline{P}$ and $\psi \in \Sigma$, 
	$\psi(p_0) \leq \psi(p_r)+r$ from the condition 2).
	If a path $C$ in $\overline{P}$ satisfying (*) has a decomposition into $A_1 + D_1 + \dots +A_n \ (A_i = (p_{i_0}, \dots , p_{r_i}))$ and $ \psi \in \Sigma$, 
	then $\psi(-\infty) \leq \psi(\infty) + \sum_{i=1}^{n}r_i \leq \rank_* \overline{P}$. 
	Hence $\fpt(\mathfrak{m}) \leq \rank_*P + 2$. 
	
	Next, we will prove that $\fpt(\mathfrak{m}) \geq \len^*C$ for some path $C$ in $\overline{P}$. 
	If we can find such a path, we will get $\fpt(\mathfrak{m}) \ge \len^*C \ge \rank_{*} \overline{P}$. 
	In general, we can calculate $\fpt(\mathfrak{m})$ by following. 
	We will define $\lambda_i$ and $\Lambda_i (i \in \mathbb{N})$ inductively as subsets of $\overline{P}$. 
	
	\begin{enumerate}
		\item $\Lambda_0 = \lambda_0 = \{\infty\}$. 
		\item $\lambda_i = \{ p \in \overline{P} \setminus \bigcup_{j=0}^{i-1} \Lambda_j \mid \exists q \in \Lambda_{i-1} \ s.t. \ p \lessdot q\}$. 
		\item $\Lambda_i = \{ p \in \overline{P} \setminus \bigcup_{j=0}^{i-1} \Lambda_j \mid \exists q \in \lambda_i \ s.t. \ p \geq q\}$. 
	\end{enumerate}
	
	Note that $\lambda_i$ is a subset of $\Lambda_i$, 
	and $\overline{P}$ is the disjoint union of $\Lambda_i$'s. 
	
	\begin{cla}
		\label{hoge}
		Suppose that $p \in \Lambda_i$ and $p' \in \Lambda_j$. 
		If $i > j$ then $p \not > p'$. 
	\end{cla}
	
	\begin{proof}
		Suppose that $p > p'$. 
		Because $p' \in \Lambda_j$, 
		we can take q $q \in \lambda_j$ such that $q < p$. 
		Since $i > j$ and $q < p$, 
		$p \not \in \bigcup_{l=1}^{j-1}\Lambda_l$, 
		and $p \in \Lambda_j$ by the definition. 
		This is a contradiction.
	\end{proof}
	
	Now let us construct function $\psi$ given by $\psi(p) = i$ if $p$ is in $\Lambda_i$.
	Then $\psi$ is in $\Sigma$ because of Claim \ref{hoge} and $\psi(-\infty) = \max\{ i \mid \Lambda_i \neq \emptyset \}$. 
	
	Next, we will find a path $C$ such that $\len^*C = \fpt(\mathfrak{m})$. 
	Let $p_0 = -\infty$. 
	Then $p_0 \in \lambda_l$ for some $l$. 
	If $p_i \in \lambda_l$, there exists $p_{i+1} \in \Lambda_{l-1}$ for which covers $p_i$.
	If $p_i \in \Lambda_l \setminus \lambda_l$, then there exist $p \in \lambda_l$ such that $p \leq p_i$, 
	and a sequence $p_i, p_{i+1}, \dots, p_j = p \ $($p_{k+1} \lessdot p_k$ for $i \leq k \leq j$) in $\Lambda_l$.
	At the end, $p_s$ becomes to $\infty$ and define a path $C = (p_0 , \dots , p_s)$ then $C$ satisfies condition (*). 
	Because it is clear by construction that if $p \in \Lambda_k$ and $p' \in \Lambda_l$, then $k \geq l$.
	If $k = l$, then $p' \leq p$. 
	This shows $p$ and $p'$ is belong to the same $D_i$. 
	If $k > l$, then $p \not > p'$ because of Claim \ref{hoge}. 
	
%
	
	Since $\len^*C$ is the number of pairs $(p_i \lessdot p_{i+1})$.
	This corresponds to the number of non-empty $\Lambda_i$ by construction of $C$, that is $n(-\infty)$.
	For this $n$ and $C$, $\rank_*\overline{P}~\leq~\len^*C~=~n(-\infty)~\leq~\fpt(\mathfrak{m})$.
\end{proof}

%

%


\section{Application and $-a(R)$ of Hibi rings}
In this section, we recall the definition of the $a$-invariant $a(R)$ and compare $c^{\mf{m}}(\mf{m})$ and $\fpt(\mf{m})$ with $-a(R)$ for Hibi rings.


First, we recall that the definition of a-invariant is
\[
a(R) = \max\{n \in \bb{Z} \mid [H_{\mf{m}}^{\dim R}(R)]_{n} \neq 0\}
\]
(see \cite{GW}). 

Bruns and Herzog computed $a(R)$ for an ASL ([BH, Theorem 1.1]). 
By their theorem, we can obtain the following fact.

\begin{fact}([BH, Theorem 1.1])
Let $R = \mathcal{R}_{k}[D]$ be the Hibi ring made by a distributive lattice $D = J(P)$, where $P$ is a finite poset. 
Then
\[
-a(R) = \rank P + 2. 
\]
\end{fact}

In particular, Theorems 2.4 and 3.7 imply the following corollary. 

\begin{cor}
Under the same notation as in Theorem 2.4, we have
\[
\fpt(\mf{m}) \le -a(R) \le c^{\mf{m}}(\mf{m}). 
\]
\end{cor}

Segre products of two polynomial rings are one of the important examples of Hibi rings.
Since the Segre product of $k[X_{1}, \ldots, X_{m}]$ and $k[Y_{1}, \ldots, Y_{n}]$ is isomorphic to the determinantal ring $k[X]/I_{2}(X)$, where $X$ is an $m \times n$ matrix whose all entries are indeterminates, we give the following corollary by Example 1.3. 

\begin{cor}
Let $k$ be a perfect field of positive characteristic, and let
$m,n \ge 2$ be integers. 
Let $R=k[X_1, \ldots, X_m], S=k[Y_1, \ldots, Y_n]$ be polynomial rings, and let $R \# S$ be the Segre product of $R$ and $S$. 
Let $\mf{m}$ be the unique graded maximal ideal of $R \# S$. 
Then
	\begin{align*}
		c^{\mf{m}}(\mf{m}) = -a(R \# S) &= \max\{m, n\},\\
		\fpt(\mf{m}) &= \min\{m, n\}. 
	\end{align*}
In particular, $c^{\mf{m}}(\mf{m}) = \fpt(\mf{m})$ if and only if $m = n$.
\end{cor}


Finally, we make general examples. 

\begin{ex}
For given integers $a \ge b\ge c \ge 1$, we can find a connected poset $P$ such that $\rank^{*} P = a, \rank P = b$ and $\rank_{*} P = c$. 
We put $a = db + r$, $e = \lceil \frac{a}{b} \rceil$ and $f = \max\{c - (b - r + 1), 0\} + 1$, where $0 \le r \le b - 1$. 
\end{ex}

\noindent Case 1: $d \ge 2$.

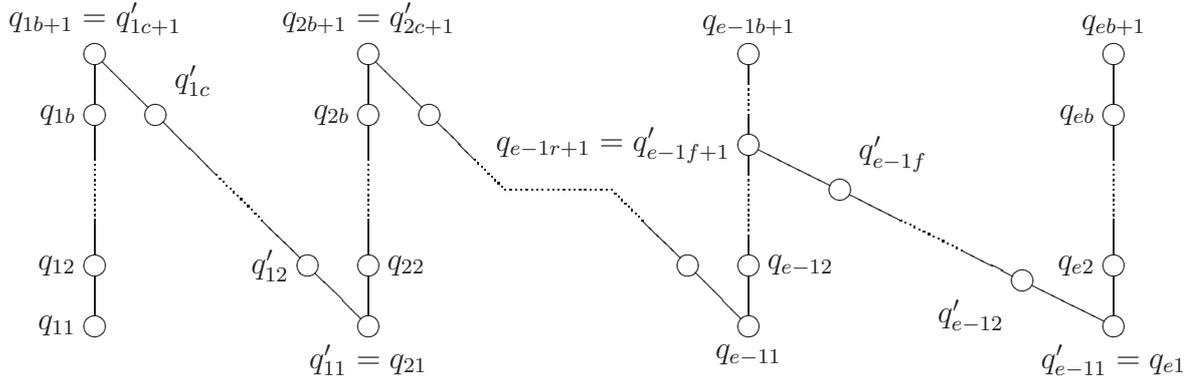
\begin{xy}
	\ar@{} (0, 0);(0, -46) *++!R{q_{11}} *\cir<4pt>{} = "A1";
	\ar@{-} "A1";(0, -38) *++!R{q_{12}} *\cir<4pt>{} = "A2";
	\ar@{-} "A2";(0, -32) = "A3";
	\ar@{.} "A3";(0, -24) = "A4";
	\ar@{-} "A4";(0, -18) *++!R{q_{1b}} *\cir<4pt>{} = "A5";
	\ar@{-} "A5";(0, -10) *++!D{q_{1b + 1} = q^{\prime}_{1c+1}} *\cir<4pt>{} = "A6";
	\ar@{-} "A6";(8, -18) *++!LD{q^{\prime}_{1c}} *\cir<4pt>{} = "B1";
	\ar@{-} "B1";(16, -26) = "B2";
	\ar@{.} "B2";(22, -32) = "B3";
	\ar@{-} "B3";(28, -38) *++!R{q^{\prime}_{12}} *\cir<4pt>{} = "B4";
	\ar@{-} "B4";(36, -46) *++!U{q^{\prime}_{11} = q_{21}} *\cir<4pt>{} = "B5";
	\ar@{-} "B5";(36, -38) *++!L{q_{22}} *\cir<4pt>{} = "C1";
	\ar@{-} "C1";(36, -32) = "C2";
	\ar@{.} "C2";(36, -24) = "C3";
	\ar@{-} "C3";(36, -18) *++!R{q_{2b}} *\cir<4pt>{} = "C4";
	\ar@{-} "C4";(36, -10) *++!D{q_{2b + 1} = q^{\prime}_{2c+1}} *\cir<4pt>{} = "C5";
	\ar@{-} "C5";(44, -18) *++!R{} *\cir<4pt>{} = "D1";
	\ar@{-} "D1";(50, -24) = "D2";
	\ar@{.} "D2";(54, -28) = "D3";
	\ar@{.} "D3";(68, -28) = "D4";
	\ar@{.} "D4";(72, -32) = "D5";
	\ar@{-} "D5";(78, -38) *++!R{} *\cir<4pt>{} = "D6";
	\ar@{-} "D6";(86, -46) *++!U{q_{e-11}} *\cir<4pt>{} = "D7";
	\ar@{-} "D7";(86, -38) *++!L{q_{e-12}} *\cir<4pt>{} = "E1";
	\ar@{-} "E1";(86, -34) = "E2";
	\ar@{.} "E2";(86, -26) = "E3";
	\ar@{-} "E3";(86, -22) *++!R{q_{e-1 r + 1} = q^{\prime}_{e-1 f + 1}} *\cir<4pt>{} = "E4";
	\ar@{-} "E4";(86, -18) = "E5";
	\ar@{.} "E5";(86, -14) = "E6";
	\ar@{-} "E6";(86, -10) *++!D{q_{e - 1 b + 1}} *\cir<4pt>{} = "E7";
	\ar@{-} "E4";(98, -28) *++!DL{q^{\prime}_{e - 1 f}} *\cir<4pt>{} = "F1";
	\ar@{-} "F1";(106, -32) = "F2";
	\ar@{.} "F2";(114, -36) = "F3";
	\ar@{-} "F3";(122, -40) *++!UR{q^{\prime}_{e - 12}} *\cir<4pt>{} = "F4";
	\ar@{-} "F4";(134, -46) *++!U{q^{\prime}_{e - 11} = q_{e1}} *\cir<4pt>{} = "F5";
	\ar@{-} "F5";(134, -38) *++!R{q_{e2}} *\cir<4pt>{} = "G1";
	\ar@{-} "G1";(134, -32) = "G2";
	\ar@{.} "G2";(134, -24) = "G3";
	\ar@{-} "G3";(134, -18) *++!R{q_{eb}} *\cir<4pt>{} = "G4";
	\ar@{-} "G4";(134, -10) *++!D{q_{eb + 1}} *\cir<4pt>{} = "G5";
\end{xy}

\noindent Case 2: $d = 1$ and $c \ge b - r$. 
Put $g = c - b + r$. 

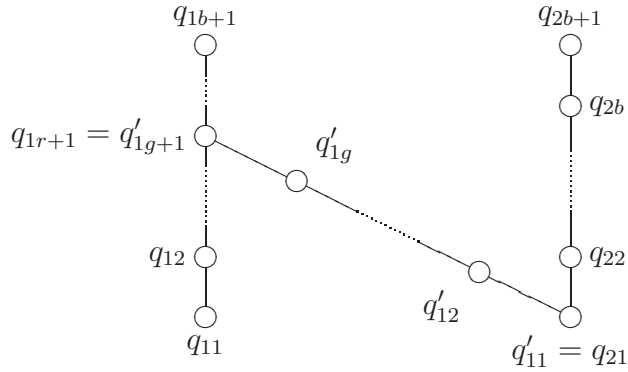
\begin{xy}
	\ar@{} (0, 0);(48, -46) *++!U{q_{11}} *\cir<4pt>{} = "A1";
	\ar@{-} "A1";(48, -38) *++!R{q_{12}} *\cir<4pt>{} = "A2";
	\ar@{-} "A2";(48, -34) = "A3";
	\ar@{.} "A3";(48, -26) = "A4";
	\ar@{-} "A4";(48, -22) *++!R{q_{1 r + 1} = q^{\prime}_{1 g + 1}} *\cir<4pt>{} = "A5";
	\ar@{-} "A5";(48, -18) = "A6";
	\ar@{.} "A6";(48, -14) = "A7";
	\ar@{-} "A7";(48, -10) *++!D{q_{1 b + 1}} *\cir<4pt>{} = "A8";
	\ar@{-} "A5";(60, -28) *++!DL{q^{\prime}_{1 g}} *\cir<4pt>{} = "B1";
	\ar@{-} "B1";(68, -32) = "B2";
	\ar@{.} "B2";(76, -36) = "B3";
	\ar@{-} "B3";(84, -40) *++!UR{q^{\prime}_{12}} *\cir<4pt>{} = "B4";
	\ar@{-} "B4";(96, -46) *++!U{q^{\prime}_{11} = q_{21}} *\cir<4pt>{} = "B5";
	\ar@{-} "B5";(96, -38) *++!L{q_{22}} *\cir<4pt>{} = "C1";
	\ar@{-} "C1";(96, -32) = "C2";
	\ar@{.} "C2";(96, -24) = "C3";
	\ar@{-} "C3";(96, -18) *++!L{q_{2b}} *\cir<4pt>{} = "C4";
	\ar@{-} "C4";(96, -10) *++!D{q_{2b + 1}} *\cir<4pt>{} = "C5";
\end{xy}

\noindent Case 3: $d = 1$ and $c < b - r$. 

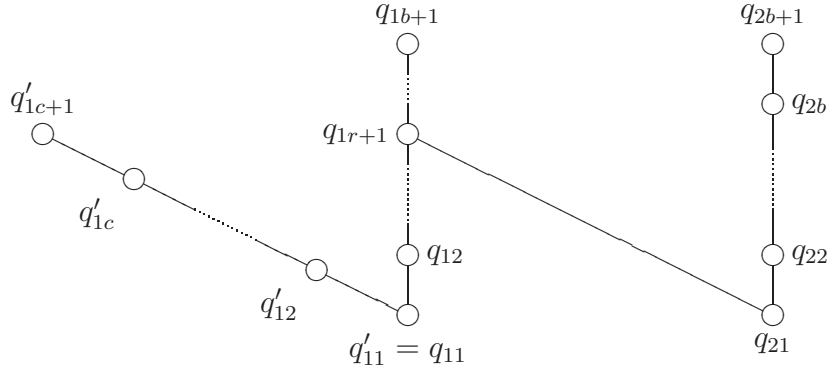
\begin{xy}
	\ar@{} (0, 0);(80, -46) *++!U{q^{\prime}_{1 1} = q_{11}} *\cir<4pt>{} = "A1";
	\ar@{-} "A1";(80, -38) *++!L{q_{12}} *\cir<4pt>{} = "A2";
	\ar@{-} "A2";(80, -34) = "A3";
	\ar@{.} "A3";(80, -26) = "A4";
	\ar@{-} "A4";(80, -22) *++!R{q_{1 r + 1}} *\cir<4pt>{} = "A5";
	\ar@{-} "A5";(80, -18) = "A6";
	\ar@{.} "A6";(80, -14) = "A7";
	\ar@{-} "A7";(80, -10) *++!D{q_{1 b + 1}} *\cir<4pt>{} = "A8";
	\ar@{-} "A5";(128, -46) *++!U{q_{21}} *\cir<4pt>{} = "B1";
	\ar@{-} "B1";(128, -38) *++!L{q_{22}} *\cir<4pt>{} = "C1";
	\ar@{-} "C1";(128, -32) = "C2";
	\ar@{.} "C2";(128, -24) = "C3";
	\ar@{-} "C3";(128, -18) *++!L{q_{2b}} *\cir<4pt>{} = "C4";
	\ar@{-} "C4";(128, -10) *++!D{q_{2b + 1}} *\cir<4pt>{} = "C5";
	\ar@{-} "A1";(68, -40) *++!UR{q^{\prime}_{1 2}} *\cir<4pt>{} = "D1";
	\ar@{-} "D1";(60, -36) = "D2";
	\ar@{.} "D2";(52, -32) = "D3";
	\ar@{-} "D3";(44, -28) *++!UR{q^{\prime}_{1 c}} *\cir<4pt>{} = "D4";
	\ar@{-} "D4";(32, -22) *++!D{q^{\prime}_{1 c + 1}} *\cir<4pt>{} = "D5";
\end{xy}


\par \vspace{2mm}
${\mb{Acknowledgement.}}$ The authors wish to thank Professor Ken-ichi Yoshida 
for many valuable comments
and his encouragement.

The second author was partially supported by Nagoya University Scholarship for Outstanding Graduate Students.


\end{document}